\documentclass{amsart}
\usepackage{amsmath,amsfonts,amssymb}
\usepackage{amsthm}
\usepackage{ifthen}
\usepackage{longtable}
\usepackage{stmaryrd}
\usepackage{array}
\usepackage{url}
\usepackage{hyperref}

\usepackage{amsmath,amssymb,amsbsy,amsfonts,amsthm,latexsym,
            amsopn,amstext,amsxtra,euscript,amscd, color}
\usepackage{hyperref}   
\usepackage[margin=4cm]{geometry}

\newcommand{\Z}{\mathbb{Z}}
\newcommand{\Q}{\mathbb{Q}}

\newcommand{\N}{\mathbb{N}}

\newtheorem{definition}{Definition}
\newtheorem*{theorem*}{Theorem}
\newtheorem{theorem}{Theorem}
\newtheorem{lemma}{Lemma}
\newtheorem{corollary}{Corollary}
\newtheorem{proposition}{Proposition}

\theoremstyle{remark}
\newtheorem{remark}{Remark}

\begin{document}
\title[Common values of linear recurrences.]{Large common values of  generalized  Ankeny-Brauer-Chowla recurrences. }

\author[Armand Noubissie]{Armand Noubissie}
\address{
{Graz University of Technology}\newline
{Institute for Analysis and Number Theory}\newline
{M\"{u}nzgrabenstrasse 36/II, 8010 Graz, Austria}}
\email{\tt armand.noubissie@tugraz.at}


\author[Robert Tichy]{Robert Tichy}
\address{
{Graz University of Technology}\newline
{Institute for Analysis and Number Theory}\newline
{Steyrergasse 30/II, 8010 Graz, Austria}}
\email{\tt tichy@tugraz.at}

\begin{abstract}
In this paper  we count the number of common values shared by two linear recurrence sequences, whose  characteristic polynomials are  a generalized Ankeny-Brauer-Chowla polynomial and  its reciprocal. More precisely, we show that these sequences have at most two sufficiently large common values. Our proof combines  Baker's theory of linear forms in logarithms of algebraic numbers with techniques from function field  theory and  from Galois theory.

\end{abstract}

\maketitle

\section{Introduction}
Consider a polynomial with two indeterminates given by
\begin{equation}\label{eq:1}
  \pi_U(X) = X(X-P_1(U)) \cdots (X- P_{q-1}(U)) + 1,
 \end{equation}
where $q$ is a positive integer greater than $3$ and $P_1(U),\ldots,P_{q-1}(U)\in \mathbb{Q}[U]$ pairwise distinct with degrees $d_j = \deg P_j(U), ~~j=1,\ldots,q-1$ such that their leading coefficients are $\geq 1$ and  $0 \leq d_1 \leq d_2 \leq \cdots \leq d_{q-3}< d_{q-2} < d_{q-1}.$  This class of polynomials seems to be investigated for the first time by Ankeny, Brauer and Chowla in  \cite{ABC}. The authors considered the polynomial $\pi_U(X)$ where $P_{q-1}(u) = u$ and $P_1, P_2, \cdots, P_{q-2}$ are constants to exhibit a large class of totally real fields $K$ with class number $h_K$ as large as possible. To the honor of the authors we will call the polynomials in \eqref{eq:1}  generalized ABC polynomials. \\

More precisely, in \cite{ABC} the authors proved the following result. Let $n \geq 2$ and $s, t$ be any two non-negative integers such that $n=s+2t$. For an arbitrary small number $\tau >0$ there exist infinitely many algebraic number fields $K$ which have $s$ real and $t$ imaginary conjugate fields such that $h_K > \vert D_K\vert ^{1/2 -\tau}$ holds for the class number $h_K$ and the discriminant  $D_K$ of the number field $K$. Sprind\v{z}uk \cite{SP} showed that this property is in some sense generic for number fields.  He obtained an upper bound for the number $N_{n,\delta}(Z)$ of non-isomorphic number fields $K$ of degree $n$ with regulator $R_K \leq Z$ satisfying the inequality $h_K \leq \vert D_K \vert^{\delta}$ for any $0 \leq \delta < 1/2$. \\
 In \cite{BL} Bilu and Luca investigated divisibility properties of class numbers. Their main argument relies on properties of ABC polynomials and some ideas of \cite{HLPT}. More precisely, for positive integers $n\geq 3$ and $l$, they have shown for any sufficiently large $X$, there are at least $cX^{\mu}$  ( $c$ some positive constant and $\mu = \frac{1}{2l(n-1)}$) pairwise non-isomorphic number fields $K$ of degree $n$ with $\vert D_K \vert \leq X$ and $h_K$ divisible by $l$. \\

It is natural to associate to the generalized ABC polynomial in \eqref{eq:1} a parametric Thue equation 
 \begin{equation}\label{eq:2}
  Y^q\pi_U(X/Y) = X(X-P_1(U)Y) \cdots (X- P_{q-1}(U)Y) + Y^q = 1.
 \end{equation}
Bombieri and Schmidt \cite{BS} proved that a Thue equation of degree $d$ admits $O(d)$ solutions. Moreover, they pointed out  that the Thue equation \eqref{eq:2} with $P_1, P_2, \cdots, P_{q-1}$  pairwise distinct integer constants has at least $d+1$ solutions namely the so-called trivial solutions $(1,0), (0,1), (P_j,1)$ with $j \in \{1, 2, \cdots, d-1\}$ which makes their bound the best possible up to a constant. In 1993, E. Thomas \cite{T} investigated equation \eqref{eq:2} in the case $P_j \in \mathbb{Z}[u]$ for  $j = 1, 2, \cdots, d-1.$  He conjectured that equation  \eqref{eq:2} has only trivial solutions if $u$ is large enough and proved it in specific cases. For example, he showed that in case $d=3,~ P_1(u)=u^k, ~P_2(u)=u^l $ where $0<k<l$, his conjecture is true with an effectively computable lower bound of $u$. Since then, Thomas conjecture has been  investigated in the literature for small degrees ($d \leq 8$). For a survey, we refer to Heuberger \cite{H1}. The author in \cite{H2} settled Thomas conjecture under a certain condition on the $P_j's,$ and some cases of this conjecture were also investigated in \cite{HLPT}. The function field analogue of this conjecture has been proved by Ziegler \cite{Z} under a weak condition on the degrees of the polynomials $P_j's$ and this study was recently extended to norm form equations by Amoroso  et al. \cite{AMS}.
This paper was motivated by the articles \cite{Z}, \cite{PT} and \cite{Pe}. In \cite{PT},  Peth\H{o}  and Tengely  studied successfully the common values of linear recurrence sequence associated to Shanks'  simplest cubics.\\
  
  Let $(A_n(u))_n$ be a linear recurrence sequence with characteristic polynomial given by a generalized ABC polynomial defined in \eqref{eq:1} and in the present paper we assume that the initial terms $A_0, \cdots, A_{q-1} \in \mathbb{Q}$. It is easy to see that the sequence $(A_n(u))_n$ can be extended to negative indices such that all members of the two-sided sequence are rational numbers. Hence the search for the integers which appear multiple times in the two - parametric sequence $(A_n(u))_n$ and $(A_{-n}(u))_{n}$ is meaningful and non-trivial. In the language of Diophantine equations this is equivalent  to find solutions of the equation 
  \begin{equation}\label{eq:3}
 \vert A_n(u) \vert = \vert A_m(u) \vert, \quad \mbox{in}\quad n, m, u \in \mathbb{Z}.
 \end{equation}
Notice that $A_{-m}(u) = B_{m}(u)$ where the characteristic polynomial of the recurrence sequence $B_m$  is the reciprocal of the characteristic polynomial of the recurrence sequence $A_m$. Before stating our main result, we first define the notion of large solutions which is crucial for the whole  paper.

\begin{definition}\label{1}
 Let $c_1, c_2$ be  effectively computable positive constants only  depending  on $P_j$ and $A_j$ where $1 \leq j \leq q-1$.   For fixed $u>c_2,$ we  say that a pair of integers $(n,m)$ is a large solution with respect to $u$  if $(n,m,u)$ is a solution to equation \eqref{eq:3}, $n \neq m$ and $\min\{\vert n \vert, \vert m \vert \} \geq c_1$.
\end{definition}
\begin{remark}
 In \cite[Theorem 1]{Pe}, Peth\H{o} considered equation \eqref{eq:3}  in the case  where  $P_{q-1}(u) = u$ and $P_1, P_2, \cdots, P_{q-2}$ are constants. Furthermore,  he had to assume that   $ P_1  \cdots  P_{q-2} = \pm 1$ and  proved in this case that \eqref{eq:3} does not have a large solution.
 \end{remark}

Our result can now be stated as follows.
\begin{theorem}\label{thm1}
 For any $u>c_2,$ the equation \eqref{eq:3} has at most $2$ large solutions.
\end{theorem}

The rest of the paper is organized as follows. In section \ref{sec:2}, we investigate the algebraic properties of the polynomial $\pi_U(X)$. More precisely, we first show that, for sufficiently large $u$, the polynomial $\pi_u(X)$ is irreducible whenever the polynomials $P_i's$ have rational coefficients. This result is based on a finiteness result for Hilbert irreducibility theorem due to M\H{u}ller in \cite{M}. Secondly, we prove that all the roots of $\pi_u(X)$ are real whenever $u$ is sufficiently large. Section \ref{sec:3} is devoted to the study of arithmetic properties of the generalized ABC recurrence sequence $(A_n(u))_n$, with particular emphasis on its growth behavior for sufficiently large $u$. In the last section,  we first derive an effective upper bound for $\max \{|n|, |m|\}$ in terms of $u$ using Baker's method of linear forms in logarithms over algebraic numbers. To eliminate the dependence on $u,$ we have to apply linear algebra tools and elementary  properties of the  logarithm function, thereby completing the proof of Theorem \ref{thm1}.

\section{Preliminaries}\label{sec:2}

    For the algebraic function $f(x)$ with Puiseux expansion $f(x)= {\displaystyle \sum_{ j=m }^{ -\infty }}  f_j x^{j/d}, \; f_j\in \mathbb{C}$ let $\lfloor f \rfloor = {\displaystyle \sum_{j=m}^0}  f_j x^{j/d}$ be its {\it polynomial part}. Clearly, if $m<0$ then the sum is empty, which means $\lfloor f \rfloor = 0$. We denote by  $\nu_U$ the valuation defined over $\overline {\mathbb{C}(U)} = {\displaystyle \bigcup_{r \geq 1 }} \mathbb{C}((U^{-1/r}))$ corresponding to the place at infinity.

\begin{lemma} \label{lem1}

Let $q\ge 2$ and $Q_U(X)  \in \mathbb{C}[U,X]$ of degree $q$ in $X$. Denote $\beta_1(U),\ldots,\beta_q(U)$ the zeroes of $Q_U(X)$ in the algebraic closure $\overline {\mathbb{C}(U)}$. Denote $\alpha_0(U),\ldots,\alpha_q(U)$ the zeroes of $T_U(X)= XQ_U(X)+p$ with some $0\not=p\in \mathbb{C}$. If $\lfloor \beta_1(U) \rfloor,\ldots,\lfloor \beta_q(U) \rfloor$ are pairwise distinct and $\nu_U(\beta_q(U))< \nu_U(\beta_j(U)) \le 0, j=1,\ldots,q-1$, then $\alpha_0(U),\ldots,\alpha_q(U)$ are pairwise distinct as well and $\lfloor \alpha_j(U) \rfloor = \lfloor \beta_j(U) \rfloor, j=1,\ldots,q$ holds after a possible permutation of the $\alpha$'s.

\end{lemma}

\begin{proof}

Assume first that $\lfloor \beta_q(U) \rfloor \not= \lfloor \alpha_j(U) \rfloor, j=0,\ldots,q$. Then $\nu_U(\beta_q(U) -\alpha_j(U))\le 0$ and equality holds only if $\nu_U(\alpha_j(U)) = \nu_U(\beta_q(U)) < 0$. Inserting $\beta_k(U),\; 1\le k\le q$ into $T_U(X)$ we get
$$
T_U(\beta_k(U)) = \beta_k Q_U(\beta_k(U))+p = p = \prod_{j=0}^q(\beta_k(U)-\alpha_j(U)),
$$
which yields
\begin{equation} \label{e:Pbeta}
 \sum_{j=0}^q \nu_U(\beta_k(U)-\alpha_j(U)) =0.
\end{equation}
Specializing $k=q,$ all summands are non-positive, hence the equality may hold only if $\nu_U(\beta_q(U)-\alpha_j(U)) = 0$ for all $j=0,\ldots,q$. As $\nu_U(\beta_q(U))<0$ this is only possible if $\nu_U(\alpha_j(U))  = \nu_U(\beta_q(U))< 0 $ for all $j=0,\ldots,q$. However, this contradicts 
\begin{equation} \label{e:Pbeta1}
\prod_{j=0}^q \alpha_j(U) = \pm p,
\end{equation}
which yields 
$$
\sum_{j=0}^q \nu_U(\alpha_j(U)) = 0.
$$
Thus there is a zero of $T_U(X)$, which will be denoted by $\alpha_q(U)$ such that, $\lfloor \alpha_q(U) \rfloor = \lfloor \beta_q(U) \rfloor$, especially $\nu_U(\alpha_q(U)) = \nu_U( \beta_q(U))<0 $.

Next we prove that there is a zero of $T_U(X)$, which will be denoted by $\alpha_0(U)$, satisfying $\lfloor \alpha_0(U) \rfloor \not= \lfloor \beta_j(U) \rfloor , j=1,\ldots,q$. Indeed, assume in the contrary that for all $j=0,\ldots,q$ there is a $1\le k(j)\le q$ such that $\lfloor \alpha_j(U) \rfloor = \lfloor \beta_{k(j)}(U) \rfloor ,\; j=1,\ldots,q$. Then $\nu_U( \alpha_j(U)) = \nu_U( \beta_{k(j)}(U) )\le 0 , \;j=1,\ldots,q$, and $\nu_U(\alpha_q(U)) = \nu_U( \beta_q(U))<0 $. Now \eqref{e:Pbeta1} yields the searched contradiction. Our second claim is correct as well. 

We prove additionally $\nu_U(\alpha_0(U))>0$. Indeed, for $0\le j\le q$ we have
$$
T_U(\alpha_j(U)) = 0 = \alpha_j(U)\prod_{k=1}^q (\alpha_j(U) - \beta_k(U)) + p,
$$ 
hence
\begin{equation} \label{e:Palpha}
\nu_U(\alpha_j(U)) + \sum_{k=1}^q \nu_U(\alpha_j(U) - \beta_k(U)) =0.
\end{equation}

We apply \eqref{e:Palpha} with $j=0$. As $\lfloor \alpha_0(U) \rfloor \not= \lfloor \beta_k(U) \rfloor , k=1,\ldots,q$, and $\nu_U(\beta_k(U))\le 0$, all, but possibly the first, summands are non-positive, but then the first summand has to be non-negative. If $\nu_U(\alpha_j(U)) = 0$ then $\nu_U(\alpha_0(U) - \beta_q(U))= \nu_U(\beta_q(U))<0$, which contradicts $\eqref{e:Palpha}$. Hence 
$\nu_U(\alpha_0(U)) > 0$, which yields $\nu_U(\alpha_0(U) - \beta_k(U))= \nu_U(\beta_k(U))$, and finally 
$$
\nu_U(\alpha_0(U)) = -\sum_{k=1}^q \nu_U(\beta_k(U)) >0.
$$
Assume now that there is $k=1,\ldots,q-1$ such that $\lfloor \beta_k(U) \rfloor \not= \lfloor \alpha_{j}(U) \rfloor ,\; j=0,\ldots,q$. This implies the  inequalities $\nu_U(\alpha_j(U)-\beta_k(U))\le 0$ for all $j=0,\ldots,q$. Equality \eqref{e:Pbeta} shows  that 
$$
\sum_{j=0}^q \nu_U(\beta_k(U)-\alpha_j(U)) =0.
$$

All summands of the left hand side are non-positive, hence equality may hold only if $\nu_U(\beta_k(U)-\alpha_j(U))=0$ for all $j=0,\ldots,q$. However this is simultaneously impossible for $j=0$ and $q$. 
Thus for all $k=1,\ldots,q-1$ there is a $0\le j(k)\le q$ such that $\lfloor \beta_k(U) \rfloor = \lfloor \alpha_{j(k)}(U) \rfloor$. This is true for $k=q$ as well. All these zeroes have non-zero polynomial parts, moreover their number is $q$ because the $\lfloor \beta_k(U) \rfloor$'s are pairwise distinct. We proved that $T_U(X)$ has the root $\alpha_0(U)$, the polynomial part of which differs from the polynomial parts of all $\beta_k(U)$'s. Altogether that are $q+1$ roots, but the number of zeroes of $T_U(X)$ is exactly the same. The Lemma is proved.
\end{proof} 

Now we specialize the result to polynomials with integer coefficients. 

\begin{corollary}\label{coro}

Let $q\ge 2$ and $Q_U(X)\in \mathbb{Z}[U,X]$ of degree $q$ in $X$. Assume that $T_U(X)$ has the properties given in Lemma \ref{lem1} and $\beta_q \in \mathbb{Q}[U]$. If $u\in \mathbb{Z}$ is large enough then $T_u(X)$ is irreducible with  dominating root $\alpha_{q}(u)$ and minorant root $\alpha_{0}(u)$.

\end{corollary}

\begin{proof}

By Lemma \ref{lem1}  we have $\nu_U(\alpha_{q}(U)) = \nu_U(\beta_{q}(U)) < \nu_U(\alpha_j(U))\le 0$ for all $j=1,\ldots q$ and $\nu_U(\alpha_0(U)) = -\sum_{k=1}^q \nu_U(\alpha_k(U))$.  Thus, for sufficiently large $u$, it follows that $\alpha_{q}(u)$ is dominant and $\alpha_0(u) = o(1)$. Therefore  $\alpha_{0}(u)$ is minorant. The irreducibility can be proved as Ankeny, Brauer and Chowla did in \cite[Lemma 1]{ABC}.

\end{proof}

By restricting the result of Lemma \ref{lem1} to polynomials with rational coefficients, we obtain the following 

\begin{lemma}\label{lem2}
Let $q\ge 2$ and $P_1(U),\ldots,P_{q-1}(U)\in \mathbb{Q}[U]$ pairwise distinct with degrees $d_j = \deg P_j(U), j=1,\ldots,q-1$. Denote $\alpha_0(U),\ldots,\alpha_q(U)$ the zeroes of $\pi_U(X)$. Assume that $0\le d_1\le \ldots \le d_{q-2}< d_{q-1}$. If $u\in \mathbb{Z}$ is large enough then $\pi_u(X)$ is irreducible over $\mathbb{Q}$ with the dominating root $\alpha_{q-1}(u)$. Moreover $\nu_U(\alpha_{j}(U)) = -d_{j}$ for all $j \in \{1, 2, \cdots, q-1\}$.
\end{lemma}

\begin{proof}

 By Lemma \ref{lem1} we have $\nu_U(\alpha_{j}(U)) = -d_{j}$ for all $j \in \{1, 2, \cdots, q-1\}$. Therefore $\nu_U(\alpha_0(U)) = {\displaystyle \sum_{j=1}^{q-1}}  d_j$. The polynomial $\pi_U$ is irreducible over $\mathbb{C}(U)$. Otherwise there exist two polynomials $P_1$ and $P_2$  such that $\pi_U(X) = P_1(X)\cdot P_2(X)$ and $P_2(\alpha_{q-1}) \neq 0$. Let $\alpha_{i_0}$ be a root of $P_2$. By definition of $\pi_U$, we have
$$\pi_U(P_{q-1}) =  \prod_{i=0}^{q-1} (P_{q-1}  - \alpha_i) = 1.$$ 
Hence it follows  that the norm of $P_{q-1}  - \alpha_{i_0}$ is a complex number; together with $P_2(\alpha_{q-1}) \neq 0$ it follows that $\nu_U(P_{q-1}) =0$ which is a contradiction. Therefore $\pi_U$ is irreducible  over $\mathbb{C}(U),$ a fortiori over $\mathbb{Q}(U)$. Combining this with the fact that  $\nu_U(\alpha_{j}(U)) = - d_{j}$ for all $j \in \{1, 2, \cdots, q-1\}$  and $P_1,\ldots,P_{q-1}$ are pairwise distinct we deduce $\mathbb{Q}(U, \alpha_0)$ is  splitting completely at infinity. Then, applying the Fundamental Equality Theorem (see for e.g. \cite[Theorem 3.11]{H}), we obtain that all the places of $\mathbb{Q}(U, \alpha_0)$ above the infinite place have relative degree  $1$ and by the definition of the degree of a place, we deduce that the degree of places above the infinite place is $1.$ From \cite[Theorem 4.7]{M}, it follows $\pi_u(X)$ is irreducible over $\mathbb{Q}$ whenever $u$ is large enough. It is well known by the general theory of algebraic functions ( see for e.g. \cite[Lemma 1]{W} for some particular cases) that  the radius of convergence of the Puiseux expansion of the roots of $\pi_U$ is positive, thus if $u$ is large enough then $\lim_{u \to \infty} \alpha_j(u) = P_j(u), j=0,\ldots,q-1$, where $P_0(u)=0$ and therefore $\alpha_{q-1}(u)$ is  dominating.

\end{proof}

\begin{remark}
Notice that the result in Corollary \ref{coro} is explicit in the sense that an effective lower bound of $u$ can be found while the result in
Lemma \ref{lem2} is not effective since it is a consequence of \cite[Theorem 4.7]{M} whose  proof is heavily  based on Siegel's  theorem which is non-effective.\\
\end{remark}

\begin{remark}
The proof of Theorem \ref{thm1} in the case where the characteristic polynomial of the sequence  $(A_n(u))_n$ is given by $T_U(X)= XQ_U(X)+1$ with $Q_U(X)\in \mathbb{Z}[U,X]$ of degree $q$ in $X$ which is separable and $\beta_q \in \mathbb{Q}[U]$ is similar to the proof given in the section below where the characteristic polynomial of the sequence  $(A_n(u))_n$ is given by a generalized ABC polynomial and so left to the reader.
\end{remark}

\begin{lemma}\label{lem3}
For sufficiently large $u$,  the roots $\alpha_0 (u), \alpha_1 (u), \cdots, \alpha_{q-1}(u)$ of the polynomial $\pi_u(X)$ are real. Moreover, $\alpha_0 (u), \alpha_{q-1}(u)$ are multiplicatively independent.
\end{lemma}
\begin{proof}
Assume that $\alpha_i(u)$ is not real for some $i \in \{0,  q-1\}$. Then the conjugate of   $\alpha_i(u)$ is different to  $\alpha_i(u)$, therefore there exists $0\leq j \leq q-1$ with $i \neq j$ such that $\overline{\alpha_i(u)} = \alpha_j(u)$,  which implies $\vert \alpha_i(u) \vert = \vert \alpha_j(u) \vert$ for sufficiently large  $u$. By replacing $T_U(X)$ with $\pi_U(X)$ in Lemma \ref{lem1}, we get a contradiction.\\
Finally assume that $\alpha_0(u)$ and  $\alpha_{q-1}(u)$  are multiplicatively dependent for some large integer $u$. Then, there is $(a,b) \in \mathbb{Z}^2\setminus \{(0,0)\}$ such that  $\alpha_0(u)^a \alpha_{q-1}(u)^b =1.$ Without loss of generality we may assume that $a\neq 0$ and $b\neq 0$. If $a > 0, b>0$ and there is a automorphism which permute $\alpha_0(u)$ and $\alpha_{q-1}(u)$ then we obtain $\alpha_0(u)^{b-a} = \alpha_{q-1}(u)^{\pm (b-a)}$. This implies the ratio of $\alpha_0(u)$ and $\alpha_{q-1}(u)^{\pm 1}$ is root of unity which, by Lemma \ref{lem2} is impossible for sufficiently large $u$. Now assume that $a > 0, b>0$ and there is no automorphism which permute $\alpha_0(u)$ and $\alpha_{q-1}(u)$. From  Galois theory, we deduce the existence of an element $\sigma$ of the Galois group associated to the splitting field of the polynomial $\pi_u(X)$ such that $\sigma (\alpha_0(u)) = \alpha_{q-1}(u)$ and  $\sigma (\alpha_{q-1}(u)) = \alpha_k(u)$ for some $k \neq 0.$ Therefore, for sufficiently large $u$, we have $$ 1 = \vert \sigma (\alpha_0(u))^a \sigma (\alpha_{q-1}(u))^b\vert = \vert \alpha_{q-1}(u)^a \vert \cdot \vert \alpha_k(u)^b\vert >1 $$ where we have used Lemma \ref{lem2} to get the last inequality. This is a contradiction. If $a > 0, ~b<0,$ then, the desired result follows easily from Lemma \ref{lem2}. By symmetry the case $a < 0, ~b \neq 0$ follows using the same strategy as above, therefore the proof is completed.
\end{proof}

\section{Properties of the generalized ABC recurrences.}\label{sec:3}

 We are now studying arithmetic properties of the generalized ABC linear recurrence sequence $(A_n(u))_n$. In the rest of this paper  we assume that at least one of $A_0, A_1, \cdots, A_{q-1}$ is non-zero and we denote by $C_1, C_2, \cdots$ effectively computable constants that depend only on the polynomials $P_1,P_2, \cdots, P_{q-1}$ and the initial values $A_0, A_1, \cdots, A_{q-1}$. From Lemma \ref{lem1}, all the roots of $\pi_u$ are different from  $zero$ for sufficiently large $u$. So the sequence $(A_n(u))_n$ can be extended uniquely for negative indexes. More precisely $(A_{-n}(u))_{n>0}$ is a linear recurrence sequence with characteristic polynomial $X^q\pi_u(1/X)$ which is the reciprocal polynomial of $\pi_u$. In what follows, we assume that $u$ is sufficiently large such that the results in Lemma \ref{lem1}, \ref{lem2}, \ref{lem3} hold. From Binet formula, there exist uniquely determined functions $a_0(u), \cdots, a_{q-1}(u) \in \mathbb{Q}(\alpha_0(u), \cdots, \alpha_{q-1}(u))$ such that 
  \begin{equation}\label{eq:4}
  A_n(u) = a_0(u)\alpha_0(u)^n + \cdots + a_{q-1}(u)\alpha_{q-1}(u)^n
 \end{equation}
holds for all $n \in \mathbb{Z}$. We denote by $(\textbf{b}_1, \cdots, \textbf{b}_q)$ the matrix having columns $\textbf{b}_1, \cdots, \textbf{b}_q.$ Further set 
\begin{itemize}
\item $\mathbb{A}= (A_0, \cdots, A_{q-1})^t;$
\item $\textbf{V}_j= (\alpha_j(u)^0, \cdots,\alpha_j(u)^{q-1}), \quad  j=1, 2, \cdots q-1;$
\item $\Delta(u) = {\rm det}(\textbf{V}_0, \cdots, \textbf{V}_{q-1});$
\item $\Delta_j(u) = {\rm det}(\textbf{V}_0, \cdots, \textbf{V}_{j-1}, \mathbb{A}, \textbf{V}_{j+1}, \cdots, \textbf{V}_{q-1}).$
\end{itemize}
The next result gives us more information about  the structure of the Binet formula of our generalized ABC recurrence sequence.

\begin{lemma}\label{lem4}
For sufficiently large $u$, $a_j(u) \neq 0$  for all  $~j \in \{0, \cdots, q-1 \}$. Moreover, they are all real.
\end{lemma}

\begin{proof}
Substituting $n= 0, \cdots q-1$ into the relation \eqref{eq:4}, we obtain a system of linear equations $$ A_j = a_0(u)\alpha_0(u)^j + \cdots + a_{q-1}(u)\alpha_{q-1}(u)^j, \quad j= 0, \cdots, q-1$$ in unknowns $a_0(u), \cdots, a_{q-1}(u),$ where the associated matrix is the Vandermonde matrix and so its  determinant $$\Delta(u) = \prod_{0 \leq i<j\leq q-1} (\alpha_i(u)-\alpha_j(u))$$ is a non-zero rational number by Lemma \ref{lem2}. By Cramer's rule we have
\begin{equation}\label{eq:5}
a_k(u) = \dfrac{\Delta_k(u)}{\Delta(u)}.
\end{equation}
It is well known from linear algebra theory, that $$\Delta_k(u) = B(\alpha_k(u))\times \mathop {\prod_{0 \leq i<j\leq q-1}}_{i,j \neq k} (\alpha_i(u)-\alpha_j(u)),$$ where the polynomial $B(x)$ is defined as $$B(x) = A_0 + A_1 x + \cdots + A_{q-1}x^{q-1}.$$
If $\Delta_k(u)=0$ for some $k \in \{0, 1, \cdots, q-1\}$ and sufficiently large $u$, then  $B(\alpha_k(u))=0$ since all the $\alpha_i(u)'s$ are pairwise distinct. By Lemma \ref{lem2}, $\alpha_k(u)$ is a root of the irreducible polynomial $\pi_u$. Since the degree of $\pi_u$ is $q$ and those of $B(x)$ is $q-1$, we get a contradiction. \\
The proof that $a_j(u)$ are real numbers for all $j=0, \cdots, q-1$ follows easily by observing that the matrix associated to $\Delta_j(u)$ has entries in the set consisting of $A_0, \cdots, A_{q-1}$ and $\alpha_i(u)^j$'s with $i,j = 0,1, \cdots, q-1,$ which are all real numbers by Lemma \ref{lem3}. 
\end{proof}

In the following lemma, we provide an explicit upper bound for $\vert \Delta_j(u) \vert$ for $0 \leq j\leq q-1.$

\begin{lemma}\label{lem5}
There exist effectively computable constants $C_1, C_2, C_3, C_4$ depending only on $P_0, \cdots, P_{q-1}$ and $A_0, \cdots, A_{q-1}$ such that, for sufficiently large $u$, 
$$C_1 u^{\sum_{i=1}^{q-1}id_i} \leq \vert \Delta (u) \vert \leq C_2 u^{\sum_{i=1}^{q-1}id_i},$$ and
$$\dfrac{1}{C_3 u^{((q-1)(q-2)-1)\sum_{}d_i}} \leq \vert \Delta_j (u) \vert \leq C_3 u^{(q-1)\sum_{}d_i} \quad \mbox{for} \quad j=0, \cdots, q-2, $$ and 
$$\dfrac{1}{C_4 u^{a_1}} \leq \vert \Delta_{q-1} (u) \vert \leq C_4 u^{(q-1)\sum_{i \neq q-1}^{} d_i},$$
where $a_1= ((q-1)(q-2)-1)\sum_{i \neq q-1}d_i + (q-1)\sum_{i} d_i$. 
\end{lemma}

\begin{proof}
By Lemma \ref{lem4}, we have $\Delta_j(u), \Delta(u) \neq 0$ for sufficiently large $u$. Moreover  
 $$
\begin{array}{lcl}
\vert \Delta(u) \vert   &= &  \prod_{0 \leq i<j\leq q-1} \vert \alpha_i(u) - \alpha_j(u) \vert  \\\\
   
                                    &= & \left(\vert \alpha_{q-1}(u)^{q-1} \vert \cdot  \prod_{j=0}^{q-2}\left | 1-\dfrac{\alpha_{j}(u)}{\alpha_{q-1}(u)} \right | \right) \cdots \left(\vert \alpha_{1}(u) \vert \cdot  \left | 1-\dfrac{\alpha_{0}(u)}{\alpha_{1}(u)} \right | \right).                                            
 \end{array} 
 $$
 Together with Lemma \ref{lem2}, we obtain $$C_1 u^{\sum_{i=1}^{q-1}id_i} \leq \vert \Delta (u) \vert \leq C_2 u^{\sum_{i=1}^{q-1}id_i}.$$ It remains to show the two last assertions of the lemma. Let $Q$ be a polynomial with rational coefficients and $a$ be a given rational number. We denote by ${\rm den}(a)$ the denominator of $a$ and ${\rm den}(Q)$ the least common multiple  of the denominators  of coefficients of Q . Put $B= {\rm lcm} ({\rm den}(P_1), \cdots, {\rm den}(P_{q-1}))$. We now show that any algebraic conjugate of $\Delta_0(u)$ has the form $\pm \Delta_j(u)$ for some $j$. $\pi_u$ is irreducible for $u$ sufficiently large by Lemma \ref{lem2}.   Hence, it follows by Galois theory that the splitting field $K_u$ of $\pi_u$ is embedded in the symmetric group $S_q$,  and the Galois group of $K_u$ acts transitively on the set of  roots of the polynomial $\pi_u$. Therefore, applying an automorphism of ${\rm Gal}(K_u/\mathbb{Q})$ to the matrix $(\mathbb{A}, \textbf{V}_1, \cdots, \textbf{V}_{q-1})$ interchanges  the column vectors. Hence, any conjugate of $\Delta_0(u)$ is $\pm \Delta_j(u)$ for some $j$. Together with the fact that ${\rm lcm} ({\rm den}(A_i),~ 0\leq i \leq q-1) \cdot B^{q(q-1)/2}\Delta_j(u)$ is  an algebraic integer for all $j=0, \cdots, q-1$, we deduce that 
 \begin{equation}\label{eq:6'}
\prod_{l=0}^{D} \vert \Delta_{j_l}(u) \vert  = N_{K_u/\mathbb{Q}}(\Delta_0(u)) \geq \dfrac{1}{{\rm lcm} ({\rm den}(A_i),~ 0\leq i \leq q-1)^D \cdot B^{Dq(q-1)/2}},
\end{equation}
where $D$ is the degree of the algebraic number $\Delta_0(u)$. Therefore,
\begin{equation}\label{eq:6}
\prod_{l=0}^{q-1} \vert \Delta_{l}(u) \vert \geq \dfrac{1}{{\rm lcm} ({\rm den}(A_i),~ 0\leq i \leq q-1)^q \cdot B^{q^2(q-1)/2}}.
\end{equation}
By Hadamard inequality, we get 
$$\vert \Delta_j(u) \vert \leq \vert \mathbb{A} \vert \cdot   \prod_{i \neq j} \vert \textbf{V}_i \vert \leq C_5 u^{(q-1)\sum_{i \neq j}^{} d_i} .$$
Combining these inequalities with relation \eqref{eq:6}, we obtain the two last assertions of the lemma.
\end{proof}

We finally need the following lemma in order to estimate how our generalized ABC recurrence sequence $(A_n(u))_n$ grows for sufficiently large $u$.

\begin{lemma}\label{lem6}
There are effectively computable constants $ C_6, u_2$ depending only on \\ $P_1, \cdots, P_{q-1}$ and $A_0, \cdots, A_{q-1}$ such that   if $n>0$ and $u>u_2$ then,
 \begin{equation}\label{eq:7}
\left | \vert A_n(u) \vert - \vert a_{q-1}(u)\alpha_{q-1}(u)^n \vert \right | < C_6 u^{(q-1)\sum_{i}d_i} u^{nd_{q-2}},
\end{equation}
and 
\begin{equation}\label{eq:8}
\left | \vert A_{-n}(u) \vert - \vert a_{0}(u)\alpha_{0}(u)^{-n} \vert \right | < C_6 u^{(q-1)\sum_{i}d_i} u^{-nd_{1}}.
\end{equation}
\end{lemma}

\begin{proof}
The Binet formula gives us the power sum representation
 $$A_m(u) = a_0(u)\alpha_0(u)^m+ \cdots + a_{q-1}(u)\alpha_{q-1}(u)^m $$ for all integers $m$. Suppose that $m=n$  for a positive integer $n.$ Then, for sufficiently large $u$, we have
 $$
\begin{array}{lcl}
\left | A_n(u) -  a_{q-1}(u)\alpha_{q-1}(u)^n \right |   &= &   \left |\sum_{j=0}^{q-2} a_{j}(u)\alpha_{j}(u)^n \right | \\\\
   
                                    &\leq & (q-1)\cdot {\rm max}_j\{\vert a_j(u) \vert\} \cdot {\rm max}_j\{\vert \alpha_j(u) \vert ^n\}. \\\\ 
                                     &\leq & C_7(q-1)\cdot u^{(q-1)\sum_{i}d_i} \cdot {\rm max}_j\{\vert \alpha_j(u) \vert ^n\}. \\\\
                                     &< & C_8 u^{(q-1)\sum_{i}d_i} \cdot u^{nd_{q-2}},
\end{array} 
 $$
 with $C_8$ effectively computable where for the third inequality we have used Lemma \ref{lem5} and for the last one, Lemma \ref{lem2}. 
If $m=-n$ with non-negative $n$, then the dominating term in the Binet formula is $a_{0}(u) \alpha_{0}(u)^{-n}$. Therefore we get 
  $$
\begin{array}{lcl}
\left | A_{-n}(u) -  a_{0}(u)\alpha_{0}(u)^{-n} \right |   &= &   \left |\sum_{j=1}^{q-1} a_{j}(u)\alpha_{j}(u)^{-n} \right | \\\\
    
                                     &\leq & C_{9}(q-1)\cdot u^{(q-1)\sum_{i}d_i} \cdot {\rm max}_j\{\vert \alpha_j(u) \vert ^{-n}\}. \\\\
                                     &<& C_{10} u^{(q-1)\sum_{i}d_i} \cdot u^{-nd_{1}},
\end{array} 
 $$
 with $C_{10}$ effectively computable where for the second inequality we have used Lemma \ref{lem5} and for the last one, Lemma \ref{lem2}.We may take $C_6 = \max\{C_8, C_{10}\}$ to get the desired result.
 \end{proof}
 
\begin{corollary}\label{cor1}
For $u$ large enough, there is an effectively  computable constant $n_3$ only depending on $P_1, \cdots, p_{q-1}$ and $A_0, \cdots, A_{q-1}$ such that the sequences $(\vert A_n(u) \vert )_n$ and  $(\vert A_{-n}(u) \vert )_n$ are strictly  increasing whenever $n \geq n_3$.
\end{corollary}

\begin{proof}
 We only prove the strict monotonicity of  $(\vert A_n(u) \vert )_n$; monotonicity of $(\vert A_n(u) \vert )_{-n}$ will follow similarly. By Lemma \ref{lem6} that $$\vert A_{n+1}(u) \vert > \vert  a_{q-1}(u)\alpha_{q-1}(u)^{n+1} \vert - C_6 u^{(q-1)\sum_{i}d_i} \cdot u^{(n+1)d_{q-2}}$$ and   $$\vert A_{n}(u) \vert < \vert  a_{q-1}(u)\alpha_{q-1}(u)^{n} \vert + C_6 u^{(q-1)\sum_{i}d_i} \cdot u^{nd_{q-2}}.$$ So, if $n$  satisfies  the inequality $\vert A_{n+1}(u) \vert < \vert A_{n}(u) \vert$, then we get 
 $$\vert a_{q-1}(u)\alpha_{q-1}(u)^{n} \vert \cdot (\vert \alpha_{q-1}(u) \vert - 1 ) < C_6 u^{(q-1)\sum_{i}d_i} \cdot u^{nd_{q-2}}(u^{d_{q-2}}+1).$$ From Lemma \ref{lem2}, we have $\vert \alpha_{q-1}(u) \vert > C_{11}u^{d_{q-2}}$  for sufficiently large $u$ and $C_{11}$ effectively computable. By Lemma \ref{lem5} and the fact that $d_{q-1}>d_{q-2}$, we obtain that $n < n_3$ with $n_3$ effectively computable. 
    The proof that $(\vert A_{-n}(u) \vert )_n$ is strictly  increasing follows by the same argument.
\end{proof}

\section{Proof of Theorem \ref{thm1}}\label{sec:4}

In this section, we may assume that $u$ is sufficiently large such that Corollary \ref{cor1} and Lemmas \ref{lem2}, \ref{lem3}, \ref{lem4}, \ref{lem5}, \ref{lem6} hold. Let $c_1, c_2$ be a constants which will be chosen at the end of this section. Fix $u>c_2$. Assume that there exist at least three solutions $(n, m)$ of equation \eqref{eq:3}  such that ${\rm min}(n, m) > c_1$. 
We split the proof into two cases: \\
- First assume that either  $n, m \geq 0$ or $n, m < 0$. In this situation, the desired result follows from the fact that the sequences  $(\vert A_n(u) \vert )_n$ and  $(\vert A_{-n}(u) \vert )_n$ are strictly  increasing whenever $n \geq n_3$ (see Corollary \ref{cor1}).\\
- Now we assume that $n \geq 0,~ m<0$. To avoid playing with the fact that $n$ and $m$ have different signs, we assume that $(n,m) \in \mathbb{N}$ satisfy the equality 
\begin{equation}\label{eq:10}
 \left | A_n(u)\right | = \left | A_{-m}(u)\right |.
 \end{equation}
The inequalities \eqref{eq:7} and \eqref{eq:8} imply $$ -C_6 u^{(q-1)\sum_{i}d_i} u^{nd_{q-2}} < \vert A_n(u) \vert - \vert a_{q-1}(u)\alpha_{q-1}(u)^n \vert < C_6 u^{(q-1)\sum_{i}d_i} u^{nd_{q-2}}$$ and  $$ - C_6 u^{(q-1)\sum_{i}d_i} u^{-md_{1}} < \vert A_{-m}(u) \vert - \vert a_{0}(u)\alpha_{0}(u)^{-m} \vert < C_6 u^{(q-1)\sum_{i}d_i} u^{-md_{1}}.$$ By relation \eqref{eq:10}, we deduce with an effective constant $C_{12}$
\begin{equation}\label{eq:c}
-C_{12} u^{(q-1)\sum_{i}d_i} u^{nd_{q-2}} < \vert a_{0}(u)\alpha_{0}(u)^{-m} \vert - \vert a_{q-1}(u)\alpha_{q-1}(u)^n \vert < C_{12} u^{(q-1)\sum_{i}d_i} u^{nd_{q-2}}.
\end{equation}
This implies 
\begin{equation}\label{eq:10'}
\left |  \left |\dfrac{a_0(u)}{a_{q-1}(u)} \right | \cdot \vert\alpha_{0}(u)^{-m} \vert \cdot \vert \alpha_{q-1}(u)^{-n} \vert - 1 \right | < u^{\frac{3}{4}n(d_{q-2}-d_{q-1})},
\end{equation}
for sufficiently large $u$. There exists an effectively computable constant $n_3$ such that $$ \frac{1}{2} \vert a_{q-1}(u)\alpha_{q-1}(u)^n \vert < \vert A_n(u) \vert < 3 \vert a_{q-1}(u)\alpha_{q-1}(u)^n \vert $$ and $$ \frac{1}{2} \vert a_{0}(u)\alpha_{0}(u)^{-m} \vert < \vert A_{-m}(u) \vert < 3 \vert a_{0}(u)\alpha_{0}(u)^{-m} \vert$$ for $n, m \geq n_3.$ Combining  with \eqref{eq:10} yields $$ m < n \dfrac{{\rm log} \vert \alpha_{q-1}(u) \vert }{{\rm log}  \vert \alpha_{0}(u)^{-1} \vert} - \dfrac{{\rm log}  \left| \frac{a_0(u)}{6a_{q-1}(u)} \right| }{{\rm log}  \vert \alpha_{0}(u)^{-1} \vert} < C_{13}n+ C_{14},$$ where for the last inequality we have used Lemma \ref{lem4} and the fact that $\vert \alpha_{0}(u) \alpha_{q-1}(u) \vert < 1$ with an effective constants $C_{13}, C_{14}$.  By putting $$\Gamma :=  1- \left |\dfrac{a_0(u)}{a_{q-1}(u)} \right | \cdot \vert\alpha_{0}(u)^{-m} \vert \cdot \vert \alpha_{q-1}(u)^{-n} \vert $$ and
\begin{equation}\label{eq:a}
\theta = \log (1-\Gamma) = \log \left |\dfrac{a_0(u)}{a_{q-1}(u)} \right | + \log  \vert\alpha_{0}(u)^{-m} \vert + \log \vert \alpha_{q-1}(u)^{-n} \vert. 
\end{equation}
Combining the relation \ref{eq:10} and the fact that $$
\begin{array}{lcl}
\left | A_n(u) -  a_{q-1}(u)\alpha_{q-1}(u)^n  - a_{q-2}(u)\alpha_{q-2}(u)^n\right |   &= &   \left |\sum_{j=0}^{q-3} a_{j}(u)\alpha_{j}(u)^n \right | \\\\
                                     &< & C_{15} u^{(q-1)\sum_{i}d_i} \cdot u^{nd_{q-3}},
\end{array} 
 $$
 and $$
\begin{array}{lcl}
\left | A_{-m}(u) -  a_{0}(u)\alpha_{0}(u)^{-m} \right |   &= &   \left |\sum_{j=1}^{q-1} a_{j}(u)\alpha_{j}(u)^{-n} \right | \\\\
    
                                     &<& C_{10} u^{(q-1)\sum_{i}d_i} \cdot u^{-md_{1}},
\end{array} 
 $$
with an effective constants $C_{10}, C_{15}$ we deduce 
\begin{equation}\label{eq:11}
\left | \Gamma +   \left |\dfrac{a_{q-2}(u)}{a_{q-1}(u)} \right | \cdot \left | \dfrac{\alpha_{q-2}(u)}{\alpha_{q-1}(u)} \right |^n\right | < u^{\frac{3}{4}n(d_{q-3}-d_{q-1})}.
\end{equation}
The following result gives  a best approximation of the logarithm of $1-\Gamma $.
\begin{lemma}\label{lem7}
For $n, u$ sufficiently large, we have
\begin{equation}\label{eq:12}
\left | \theta  -  \left |\dfrac{a_{q-2}(u)}{a_{q-1}(u)} \right | \cdot \left | \dfrac{\alpha_{q-2}(u)}{\alpha_{q-1}(u)} \right |^n\right | < 
\max \left( u^{\frac{3}{2}n(d_{q-2}-d_{q-1})}, u^{\frac{3}{4}n(d_{q-3}-d_{q-1})} \right).
\end{equation}
\end{lemma}

\begin{proof}
 First let us note that 
 \begin{equation}\label{eq:log-2nd-term-est}
 \frac{1}2 x^2<-x-\log(1-x)<2x^2
 \end{equation}
 provided that $|x|<0.5$. This relation \ref{eq:log-2nd-term-est}  is a direct consequence of Taylor’s theorem with Cauchy and Lagrange remainders, respectively. From the inequality \eqref{eq:10'}, we deduce that $\vert \Gamma  \vert < 1/2$ as $d_{q-1}> d_{q-2}$.
 Since $\theta = \log (1- \Gamma)$, we obtain by applying \eqref{eq:11} that 
 $$\frac{1}2 \Gamma ^2 - u^{\frac{3}{4}n(d_{q-3}-d_{q-1})}<\left |\dfrac{a_{q-2}(u)}{a_{q-1}(u)} \right | \cdot \left | \dfrac{\alpha_{q-2}(u)}{\alpha_{q-1}(u)} \right |^n - \theta < 2 \Gamma^2+ u^{\frac{3}{4}n(d_{q-3}-d_{q-1})}.$$
This together with relation  \eqref{eq:10'} proves the statement of the lemma.

\end{proof}
 
 The following result due to Matveev \cite{Mat} will be an essential tool to compute an upper bound of $n$ and $m$ which will depend on $u$. We start by defining the Weil height of an algebraic number $\eta\neq 0$  of degree $\delta$. Suppose that
$$a_\delta(X-\eta^{(1)})\dots (X-\eta^{(\delta)}) \in \Z[X]$$
is the minimal polynomial of $\eta$. Then the absolute logarithmic Weil height is defined by
$$h(\eta)=\frac 1\delta \left(\log |a_\delta|+\sum_{i=1}^\delta \max\{0,\log|\eta^{(i)}|\}\right).$$
In the case that $\eta$ is a rational number, say $\eta=P/Q \in \Q$ with coprime integers $P$ and $Q,$ we have $h(P/Q)=\max\{\log |P|,\log |Q|\}$.
We also note that algebraic conjugates have the same height. 
Let us also recall the following well-known property of the height 

\begin{proposition}\label{prop}
For algebraic numbers $\xi,\zeta$ and an  integer $\ell$ we have :
\begin{equation*}
\begin{split}
h(\xi \pm \zeta) \leq & \ h(\xi)+h(\zeta)+\log{2},  \\
h(\xi\zeta^{\pm 1}) \leq & \ h(\xi)+h(\zeta),  \\
h(\xi^\ell)=& \ |\ell|h(\xi) \qquad \text{for} \; \ell \in \Z.
\end{split}
\end{equation*}
\end{proposition}

With this basic notion of height we can state  Matveev's result \cite{Mat} on lower bounds for linear forms in logarithms.

\begin{lemma}\label{lemM}
  Denote by $\eta_1$, \dots, $\eta_t$ algebraic numbers, not $0$ nor $1$,
  by $\log\eta_1$, \dots, $\log\eta_t$ determinations of their logarithms,
  by $D$ the degree over $\Q$ of the number field $K =
  \Q(\eta_1,\ldots,\eta_t)$, and by $b_1$, \dots, $b_t$ rational
  integers. Furthermore let $\kappa=1$ if $K$ is real and $\kappa=2$
  otherwise. For all integers $j$ with $1\leq j\leq t$ choose
  \begin{equation*}
    A_j\geq \max\{D h(\eta_j), |\log\eta_j|, 0.16\},
  \end{equation*}
  and set
  \begin{equation*}
    E=\max \{1\} \cup \{|b_j| A_j /A_t\: :\: 1\leq j \leq n \}.
  \end{equation*}
  Assume that
  \begin{equation*}
    \Lambda:=b_1\log \eta_1+\cdots+b_t\log \eta_t\not=0.
  \end{equation*}
  Then
  \begin{equation*}
    \log |\Lambda|
    \geq -\tilde C(t,\kappa)\max\{1,t/6\} C_0 W_0 D^2 \Omega
  \end{equation*}
  with
  \begin{gather*}
    \Omega=A_1\cdots A_t, \\
    \tilde C(n,\kappa)= \frac {16}{t!\, \kappa} e^t(2t +1+2 \kappa)(t+2)
    (4(t+1))^{t+1} \left( \frac 12 en\right)^{\kappa}, \\
    C_0= \log\left(e^{4.4t+7}t^{5.5}D^2 \log(eD)\right),
    \quad W_0=\log(1.5eED \log(eD)).
  \end{gather*}
\end{lemma}

We will apply Lemma \ref{lemM} to $\theta = \log (1- \Gamma)$. In order to apply such lemma we have to show that $\theta \neq 0$. In  contrary, assume that $\theta =0.$ Then,  by definition of $\theta$, we have $$ \dfrac{a_0(u)}{a_{q-1}(u)}  \cdot \alpha_{0}(u)^{-m} \cdot \alpha_{q-1}(u)^{-n} = \pm 1,$$   which implies 
\begin{equation}\label{eq:14} 
\Delta_{0}(u) \cdot \alpha_{0}(u)^{-m} \cdot \alpha_{q-1}(u)^{-n} = \pm \Delta_{q-1}(u). 
\end{equation}
If there is a automorphism which permute $\alpha_0(u)$ and $\alpha_{q-1}(u)$ then he interchanges $\Delta_{0}(u)$ and $\Delta_{q-1}(u)$  therefore
\begin{equation}\label{eq:14'} 
\Delta_{q-1}(u) \cdot \alpha_{q-1}(u)^{-m} \cdot \alpha_{0}(u)^{-n} = \pm \Delta_{0}(u). 
\end{equation}
 Together with \ref{eq:14} imply the ratio of $\alpha_0(u)$ and $\alpha_{q-1}(u)^{\pm 1}$ is root of unity which, by Lemma \ref{lem2} is impossible for sufficiently large $u$. If  there is no  automorphism which permute $\alpha_0(u)$ and $\alpha_{q-1}(u)$ then
 it follows by Galois theory that the splitting field $K_u$ of $\pi_u$ is embedded in the symmetric group $S_q$,  and the Galois group of $K_u$ acts transitively on the set of  roots of the polynomial $\pi_u$ therefore  there is an element $\sigma$ of the Galois group associated to the splitting field of the polynomial $\pi_u(X)$ such that $\sigma (\alpha_0(u)) = \alpha_{q-1}(u)$ and  $\sigma (\alpha_{q-1}(u)) = \alpha_k(u)$ for some $k \neq 0.$
 By applying this automorphism to relation \eqref{eq:14} gives
$$\Delta_{i_1}(u) \cdot \alpha_{q-1}(u)^{-m} \cdot \alpha_{k}(u)^{-n} = \pm \Delta_{i_0}(u)$$ for some $i_0,  i_1 \in \{0, \cdots, q-1\}$. From Lemma \ref{lem2} and  choosing $u$ large enough, we obtain that $n< C_{16}$ where the constant $C_{16}$ is chosen large  and depends only on $P_1, \cdots, P_{d-1}$ and $A_0, \cdots, A_{d-1}$. Since $n$ is chosen large enough, we  get a contradiction. \\
Now,  Lemma \ref{lemM} is applicable. We set $K = \mathbb{Q}(\alpha_0(u), \cdots, \alpha_{q-1}(u)); ~ D \leq q!, ~t=3$ and $$\eta_1= \left |\dfrac{a_{0}(u)}{a_{q-1}(u)} \right |,~ \eta_2= \vert \alpha_0(u) ^{-1} \vert ,~ \eta_3 = \vert \alpha_{q-1}(u) ^{-1} \vert ,$$ also $$b_1 = 1,~ b_2=m, ~b_3=n \quad \mbox{and} \quad  E\leq C_{13}n+ C_{14} < C_{17}n$$ for $n$ large enough. By using the fact that $\alpha_{q-1}(u), \alpha_{0}(u)$ are roots of $\pi_u$ and Proposition \ref{prop} we obtain  $$ h(\alpha_0(u)) \leq C_{19} \log u \quad \mbox{and} \quad h(\alpha_{q-1}(u)) \leq C_{18} \log u, $$ where $C_{19}, C_{18}$ are chosen effectively. On the other hand, using Lipschiz formula for the determinant and the Weil height properties  in Proposition \ref{prop} we get $$ h\left( \dfrac{a_{0}(u)}{a_{q-1}(u)} \right) =h\left( \dfrac{\Delta_{0}(u)}{\Delta_{q-1}(u)} \right) \leq  C_{20} \log u,$$ where $C_{20}$ can be chosen effectively. we  apply Lemma \ref{lemM} and obtain with effective constant $C_{20}$
\begin{equation}\label{eq:15}
\log |\theta|> - C_{21}\log\left(C_{17}n\right)(\log u)^3.
\end{equation}
Due to  relation \eqref{eq:10} and a standard property of the logarithm function we have that $$\vert \theta \vert = |\log (1-\Gamma )| \leq 2 |\Gamma| \leq   2u^{\frac{3}{4}n(d_{q-2}-d_{q-1})}.$$ Together with \eqref{eq:15} and \cite[Lemma 2.4]{AN}, we finally obtain $$ n< C_{22} \log ^2u \cdot \log \log u,$$ where $C_{22}$ is effectively computable. So far, we have shown the following result.

\begin{proposition}\label{prop2}
 If $(n,m)\in \N^2$ is a large solution to \eqref{eq:12}, then 
 $$
 n<C_{22} \log ^2 u \cdot \log \log u,
 $$
 and 
 $$
 m<C_{23} \log ^2 u \cdot \log \log u,
 $$
  where $C_{22}$ and $C_{23}$ are effectively computable constants.
\end{proposition}

Let us assume that \eqref{eq:10} has three large solutions $(n_1,m_1)$, $(n_2,m_2)$ and $(n_3,m_3)$, then without loss of generality we may assume that $n_1<n_2<n_3$, hence $2\leq m_1<m_2<m_3$. From \eqref{eq:a} we obtain the linear system of equations
\begin{align*}
 n_1\log |\alpha_{q-1}(u)^{-1}| + 1\cdot \log  \left |\dfrac{a_{0}(u)}{a_{q-1}(u)} \right |+ m_1\log |\alpha_{0}(u)^{-1}|  &=\theta_{1}\\\\
 n_2\log |\alpha_{q-1}(u)^{-1}| + 1\cdot \log  \left |\dfrac{a_{0}(u)}{a_{q-1}(u)} \right |+ m_2\log |\alpha_{0}(u)^{-1}|  &=\theta_{2}\\\\
n_3\log |\alpha_{q-1}(u)^{-1}| + 1\cdot \log  \left |\dfrac{a_{0}(u)}{a_{q-1}(u)} \right |+ m_3\log |\alpha_{0}(u)^{-1}| &=\theta_{3}.
\end{align*}
By Cramer's rule we get by writing
$$\Delta=\left(\begin{array}{ccc} n_1 & 1 & m_1\\ n_2 & 1 & m_2\\ n_3 & 1 & m_3 \end{array}\right) \qquad \Delta_3=\left(\begin{array}{ccc} n_{1} & 1 & \theta_1\\ n_{2} & 1 & \theta_2\\ n_{3} & 1 & \theta_3 \end{array}\right)
$$
the equation
\begin{equation}\label{eq:16}
|\Delta| \log |\alpha_{0}(u)^{-1}|= |\Delta_3|.
\end{equation}

We distinguish between the case that $ \Delta  =0$ (the exceptional case) and the case that $\Delta  \neq 0$ (the generic case). We are first dealing with the generic case. Note that in this case we clearly have $|\Delta|\geq 1$, hence 
$\log |\alpha_{0}(u)^{-1}| \leq |\Delta_3|.$

Let us estimate $\Delta_3$. We have
$$|\Delta_3|=|\theta_{1}n_2+\theta_{3}n_1+\theta_{2}n_3-\theta_{3}n_2-\theta_{1}n_3-\theta_{2}n_1|<2(|\theta_{1}|+|\theta_{2}|+|\theta_{3}|)n_3.$$
 Note that we also have $$|\theta|<2|\Gamma| \leq   2u^{\frac{3}{4}n(d_{q-2}-d_{q-1})}.$$ Therefore we get
$$ \log u <\log |\alpha_{0}(u)^{-1}|\leq  |\Delta_3| \leq 6n_3u^{\frac{3}{4}n_1(d_{q-2}-d_{q-1})},$$
which yields the estimate $\log n_3>C_{24} n_1 \log u$ where $C_{24}$ is effectively computable. Together with  Proposition \ref{prop2} we deduce that
$$\log u < C_{25} \log \log u $$
provided that   $u <C_{26}$ with an effectively computable constants $C_{26}, C_{25}$ which only depend  on $P_1, \cdots, P_{d-1}$ and $A_0, \cdots, A_{d-1}.$ Therefore  Theorem \ref{thm1} is proven in this case.\\

We consider now the case that $ \Delta =0$. Then  $$|\Delta| \log |\alpha_{0}(u)^{-1}|= 0=|\Delta_3|$$ with 
$$\Delta_3=\left(\begin{array}{ccc} n_{1} & 1 & \theta_1\\ n_{2} & 1 & \theta_2\\ n_{3} & 1 & \theta_3 \end{array}\right)$$ by 
Cramer's rule. Hence, we get

 \begin{align*}
 0&=\det\left(\begin{array}{ccc} n_1 & 1 & \theta_{1}\\ n_2 & 1 & \theta_{2} \\  n_3 & 1 & \theta_{3} \end{array}\right)\\
 &=\det\left(\begin{array}{ccc}  n_1-n_2 & 0 & \theta_{1}-\theta_{2}\\ n_2 & 1 & \theta_{2} \\ n_3-n_2 & 0 & \theta_{3}-\theta_{2} \end{array}\right)\\
 &= (n_1-n_2)(\theta_{3}-\theta_{2})-(n_3-n_2)(\theta_{1}-\theta_{2}),
 \end{align*}
 and thus
 \begin{equation}\label{eq:b}
  \frac{n_2-n_1}{n_3-n_2}=\frac{\theta_{1}-\theta_{2}}{\theta_{2}-\theta_{3}}.
 \end{equation}
 We take $\epsilon > 0$ relatively small enough such that 
 $$(1-\epsilon)  \left |\dfrac{a_{q-2}(u)}{a_{q-1}(u)} \right | \cdot \left | \dfrac{\alpha_{q-2}(u)}{\alpha_{q-1}(u)} \right |^n < |\theta| < (1 + \epsilon) \left |\dfrac{a_{q-2}(u)}{a_{q-1}(u)} \right | \cdot \left | \dfrac{\alpha_{q-2}(u)}{\alpha_{q-1}(u)} \right |^n.$$
This choice of $\epsilon$ is possible by  Lemma \ref{lem7} and the fact that $d_{q-3}<d_{q-2}$. Together with relation \eqref{eq:b} implies

\begin{align*}
 n_2-n_1 &> \frac{n_2-n_1}{n_3-n_2}=\frac{\theta_{1}-\theta_{2}}{\theta_{2}-\theta_{3}}\\
 &>\dfrac{\left | \dfrac{\alpha_{q-2}(u)}{\alpha_{q-1}(u)} \right |^{n_1}(1-\epsilon)-\left | \dfrac{\alpha_{q-2}(u)}{\alpha_{q-1}(u)} \right |^{n_2}(1+\epsilon)}{\left | \dfrac{\alpha_{q-2}(u)}{\alpha_{q-1}(u)} \right |^{n_2}(1+\epsilon)- \left | \dfrac{\alpha_{q-2}(u)}{\alpha_{q-1}(u)} \right |^{n_3}(1-\epsilon)}\\
 &>\dfrac 12 u^{n_2-n_1},
\end{align*}
where for the last inequality, we have used  Lemma \ref{lem2}. This inequality combined with $n_2 > n_1$ provide  $u <C_{27}$ with an effectively computable constant $C_{27}$, which only depends  on $P_1, \cdots, P_{d-1}$ and $A_0, \cdots, A_{d-1}$.  This completely proves Theorem \ref{thm1}.

\section*{Acknowledgement}
The authors are grateful to Attila Peth\H{o} for carefully reading the manuscript and for helpful discussions about Lemma \ref{lem1} and Corollary \ref{coro} of this  paper.  The first author was funded in whole or in part by the Austrian Science Fund (FWF), Grant-DOI 10.55776/ESP1571325 and the secound author was supported  by the  Austrian-French Grant I-6750.

\bibliographystyle{plain}

\end{document}